\documentclass[12pt,a4paper,oneside]{amsart}

\usepackage{amsfonts, amsmath, amssymb, amsthm, amscd}
\usepackage{hyperref}
\hypersetup{
  colorlinks   = true, %Colours links instead of ugly boxes
  urlcolor     = blue, %Colour for external hyperlinks
  linkcolor    = blue, %Colour of internal links
  citecolor   = red %Colour of citations
}
\usepackage{anysize}
\marginsize{2cm}{2cm}{1.5cm}{1.5cm}

\usepackage{graphicx}

\newtheorem{theorem}{Theorem}[section]
\newtheorem{lemma}[theorem]{Lemma}

\newtheorem{conjecture}[theorem]{Conjecture}
\newtheorem{corollary}[theorem]{Corollary}

\theoremstyle{definition}
\newtheorem{definition}[theorem]{Definition}

\theoremstyle{remark}
\newtheorem{remark}[theorem]{Remark}

\newtheorem{example}[theorem]{Example}

\numberwithin{equation}{section}

\DeclareMathOperator{\vol}{vol}

\DeclareMathOperator{\supp}{supp}

\newcommand{\const}{{\rm const}}
\newcommand{\inte}{\mathop{\rm int}}

\renewcommand{\Re}{\mathop{\rm Re}\nolimits}
\renewcommand{\Im}{\mathop{\rm Im}\nolimits}

\renewcommand{\epsilon}{\varepsilon}
\renewcommand{\phi}{\varphi}
\renewcommand{\kappa}{\varkappa}

\newcounter{fig}
\newcommand{\f}{\refstepcounter{fig} Fig. \arabic{fig}. }

\begin{document}

\title{Bang's problem and symplectic invariants}

\author{Arseniy~Akopyan{$^\spadesuit$}}

\address{{$^\spadesuit$}Institute of Science and Technology Austria (IST Austria), Am Campus 1, 3400 Klosterneuburg, Austria}

\email{akopjan@gmail.com}

\author{Roman~Karasev{$^\clubsuit$}}

\address{{$^\clubsuit$}Dept. of Mathematics, Moscow Institute of Physics and Technology, Institutskiy per. 9, Dolgoprudny, Russia 141700}

\address{{$^\spadesuit$}{$^\clubsuit$}Institute for Information Transmission Problems RAS, Bolshoy Karetny per. 19, Moscow, Russia 127994}

\email{r\_n\_karasev@mail.ru}
\urladdr{http://www.rkarasev.ru/en/}

\author{Fedor~Petrov{$^\diamondsuit$}}

\address{{$^\diamondsuit$}Saint-Petersburg Department of the Steklov Mathematical Institute, nab. Fontanki 27, Saint-Petersburg, Russia 191023}

\email{fedyapetrov@gmail.com}

\thanks{{$^\spadesuit$} Supported by People Programme (Marie Curie Actions) of the European Union’s Seventh Framework Programme (FP7/2007–2013) under REA Grant Agreement No. [291734]}

\thanks{{$^\clubsuit$} Supported by the Federal professorship program grant 1.456.2016/1.4 and the Russian Foundation for Basic Research grants 18-01-00036 and 19-01-00169}

\thanks{{$^\diamondsuit$} Supported by the Russian Foundation for Basic Research grant 17-01-00433}


\begin{abstract}
We consider the Tarski--Bang problem about covering of convex bodies by planks. The results of this kind give a lower bound on the sum of widths of planks (regions between a pair of parallel hyperplanes) covering a given convex body.

Previously we have applied some notions of symplectic geometry to study convex bodies, and here we show that the symplectic techniques may be useful in this problem as well. We are able to handle some particular cases with the symplectic techniques, and show that the general cases would follow from a certain ``subadditivity conjecture'' in symplectic geometry, motivated by the results of K.~Ball. We also prove several related results by more elementary methods.
\end{abstract}

\maketitle

\section{Introduction}

\subsection{The Moese--Tarski--Bang problem}

We start from recalling the classical problem attributed to Alfred~Tarski and Th{\o}ger~Bang and the known results on this problem, in particular those of Keith~Ball, that give motivation to the whole discussion in this text.

The earliest version of this problem appeared when Tarski studied~\cite{tarski1931,tarski1932} certain degree of equivalence $\tau(x)$ of a unit square $Q$ and a rectangle $P$ of size $x\times \frac{1}{x}$, defined as the smallest number of parts one has to cut the rectangle into to assemble the square from the parts. To solve a particular case of this problem and show that $\tau(n)=n$ for natural numbers $n$, Henryk~Moese~\cite{moese1932} inscribed a disk $K$ into $Q$ and noticed that this disk cannot be covered by less than $n$ parts $P_i$ of $P$. The solution used the trick of projecting the sphere in $\mathbb R^3$ onto $K$ and counting the areas of the preimages of $P_i$ on the sphere. This gave the solution of what was called later ``the Bang problem'' for the round disk $K$ and the Euclidean norm.

Bang had~\cite{bang1951} a different (non-volumetric) solution of the more general problem: If a convex body $K\subset\mathbb R^n$ is covered by planks $P_1,\ldots, P_m$ (a plank is a set bounded by a pair of parallel hyperplanes) then the sum of Euclidean widths of the planks is at least the Euclidean width of $K$. Bang also conjectured~\cite{bang1951} that whenever a convex body $K$ is covered by planks $P_1, \ldots, P_m$, the sum of relative widths of the planks is at least $1$. Here the \emph{relative width} is the width of $P_i$ in the norm with the unit ball $K-K$ (the symmetrization of $K$), and this version would certainly imply the original result of Bang.

%Igor Pak liked the phrase:
%The best to date result on Bang's conjecture belongs to Ball~\cite{ball1991}, 
To date, the best result on Bang's conjecture belongs to Ball~\cite{ball1991}, who established it for all centrally symmetric convex bodies $K$. For non-symmetric bodies the problem remains open.

There is essentially one general approach to Bang's problem known so far, designed by Bang himself. For any plank $P_i$, we take a pair of points on its bounding hyperplanes at which the distance between the bounding hyperplanes (the width of the plank) is attained, call such two-point set $I_i$. The first easy step is to show that the Minkowski sum 
\[
X = I_1+\dots + I_m = \{p_1+\dots + p_m : p_1 \in I_1,\ldots, p_m\in I_m\}
\] 
can be translated to fit into any given convex body of minimal width 
\[
w(K) \ge \sum_{i=1}^m w(P_i).
\] 

The second step, the main lemma of Bang, asserts that at least one point of this Minkowski sum $X$ is not covered by the interiors of $P_i$. This only works in the Euclidean case, lifting the dimension allows to consider planks $P_i$ and sets $I_i$ centered at the origin, and then it is relatively easy to see that the point from $X$ with the largest distance from the origin is not covered by the interiors of $P_i$ (see how this simple idea works in the recent work \cite{polyajiang2017} on a similar problem). After a simple approximation argument to pass from the interiors of the planks to closed planks, this lemma immediately proves the Euclidean case of Bang's problem. This construction together with several other technical tools was also used in Ball's proof of the general symmetric case \cite{ball1991}. In~\cite{ball2001} Ball proved another version of this problem for complex vector spaces, which we also discuss in Section~\ref{section:cylinders}.

In contrast to the general case, the approach of Moese to the cases of dimension $2$ and $3$ is volumetric. The crucial observation is that for a plank $P_i$ the area of its intersection with the round two-sphere $S^2\subset\mathbb R^3$ is proportional to the plank width. It is also known that the volumetric approach fails in larger dimensions.

\subsection{Symplectic tools}

In this paper we are going to propose another ``quantitative'' approach to the Bang conjecture based on certain invariants of symplectic manifolds, known as symplectic capacities, introduced by Helmut Hofer and Eduard Zehnder (see their nice book~\cite{hofer1994}), with first nontrivial examples given previously by Mikhail Gromov~\cite{grom1985}. This approach has already proved to be useful in the intersection of symplectic and convex geometry in~\cite{aao2012,aaok2013,apbtz2013}, and allows either to solve a problem in convex geometry by symplectic methods or provides a good intuition to pose the ``right questions'' in convex geometry. 

This time we are going to do the opposite: Use the knowledge on the convex problem to pose the corresponding ``right problem'' in symplectic geometry. The central theme of our discussion if the following (rather imprecise at this point) conjecture about symplectic capacities: 

\begin{conjecture}
\label{conjecture:subadditivity}
If a convex body $K\subset \mathbb C^n$ is covered by a finite set of convex bodies $\{K_i\}$ then, for some symplectic capacity $c(\cdot)$ (the cases of interest are the Hofer--Zehnder capacity and the displacement energy),
\[
\sum_i c(K_i) \ge c(K).
\]
\end{conjecture}

In the moment such a conjecture seems rather unmotivated and the purpose of this text is to show that it has direct relation to the Bang problem. In particular, the known results by Ball~\cite{ball1991,ball2001} follow from this conjecture (and even its weaker form, Conjecture~\ref{conjecture:cylinders} explained below). Therefore positive results on this conjecture would be useful, because Ball's proofs are rather technical. The Bang conjecture for not necessarily centrally symmetric $K$ does not follow from this conjecture, but a worse estimate of the sum of widths, tending to $1/2$ for high dimensions, would follow from an appropriate version of Conjecture~\ref{conjecture:cylinders}. 

Theorem \ref{theorem:billiards-symmetrization} and Corollary \ref{corollary:weak-parallel-bang} present our partial results on the Bang problem. However, here we are only able to handle these results in a particular case of the Bang problem with ``almost parallel planks'' (see the explanation of this term in Theorem \ref{theorem:almost-par}), the general case being dependent on the subadditivity conjecture. 

\subsection{Organization of the paper}

In Sections \ref{section:symplectic}--\ref{section:cylinders} we explain how the results of Ball are expressed in terms of subadditivity of certain symplectic capacities. In particular, in Section \ref{section:billiards} we establish the estimate (Theorem \ref{theorem:billiards-symmetrization}) 
\[
c_{HZ}(K\times (K-K)^\circ) \ge 1 + \frac{1}{n},
\]
that extends the work of \cite{aaok2013,abksh2014} on lower bounds for symplectic capacities and close billiard trajectories of Minkowski billiards.

In Section \ref{section:subadd} we provide examples showing that the convexity assumption seems crucial in the subadditivity conjecture, prove this conjecture for the simple case of splitting the ball in $\mathbb C^n$ by a hyperplane cut (already generalized in a parallel work \cite{haimkislev2017}), and discuss other evidence of subadditivity.

In Section \ref{section:oscillation} we bound from below the oscillation of a function on a convex set given its differential is bounded from below. A version of this is obtained by a simple application of known symplectic results without any conjectural assumptions, but eventually it turns our that the optimal result of this kind (Theorem \ref{theorem:diff-general}) is proved without using any symplectic technique.

In Section \ref{section:two-dir} we establish another particular case of the Bang problem with elementary means, and in Section \ref{section:fractional} we discuss fractional versions of the Bang problem as well as covering of the Euclidean ball by Euclidean cylinders. Here symplectic techniques do not seem to give much.

\subsection*{Acknowledgments} 
The authors thank Yaron Ostrover for useful discussions and numerous remarks, Leonid~Polterovich for the useful Example \ref{example:curved-cylinders}, the unknown referees for lots of useful remarks, Wac{\l}aw Marzantowicz and Jakub Byszewski for scanning the old Polish papers~\cite{moese1932,tarski1931,tarski1932} for us.

\section{Displacing a Lagrangian product}
\label{section:symplectic}

Let us demonstrate the relation of the Bang conjecture to some notions of symplectic geometry. Denote by $V$ the ambient vector space of a convex body $K$, and let $V^*$ be its dual; we remind that a convex body is a compact convex set with nonempty interior. The reader may safely assume $V = V^* = \mathbb R^n$, but in order to use the symplectic viewpoint further we emphasize the duality and use the canonical bilinear product $\langle\cdot,\cdot\rangle : V\times V^*\to\mathbb R$.

Consider some norm $\|\cdot\|$ on $V$ whose unit ball we denote by $B$. Let $\|\cdot\|_*$ be the dual norm on $V^*$, its unit ball is the polar $B^\circ$ of $B$. In the Bang conjecture the natural choice of the norm $\|\cdot\|$ is the norm with the unit ball 
\[
K-K := \{x - y : x\in K,\ y\in K\},
\] 
but we do not restrict ourselves and allow arbitrary norms, unless otherwise stated. 

We always assume that the norms $\|\cdot \|$ and $\|\cdot\|_*$ are sufficiently smooth. For the Bang conjecture this is not a problem, since a standard approximation argument allows to approximate any convex body (the unit ball $B$ in this particular case) by an infinitely smooth and strictly convex body contained in it (or containing it). We also assume that $K$ has sufficiently smooth boundary when this is needed in the argument.

Our idea is to start from a covering of $K$ by planks $P_1,\ldots, P_m$ with widths $w_1,\ldots, w_m$ and show that a symplectic capacity of the convex body $K\times B^\circ\subset V\times V^*$ is bounded in terms of $\sum_i w_i$. Let us recall that a \emph{plank} is a closed region between a pair of parallel hyperplanes in $V$ and its \emph{width} is the distance between the hyperplanes in the norm $\|\cdot\|$ we work with. For arbitrary convex body $K\subset V$, we call the \emph{(minimal) width} of $K$ the minimal width of a plank containing $K$.

First of all, the space $V\times V^*$ is, in more general terms, the cotangent space of the manifold $V$. The cotangent space always inherits the canonical symplectic structure, which in this particular case is given by the formula
\[
\omega((x_1, y_1), (x_2, y_2)) = \langle x_1, y_2 \rangle - \langle x_2, y_1 \rangle.
\]

Consider the convex body $K\times B^\circ\subset V\times V^*$, which is called a \emph{Lagrangian product} in \cite{aao2012} and subsequent works, and try to figure out how the assumption that $K$ is covered with a set of planks of given total width bounds a symplectic capacity of $K\times B^\circ$. In this problem it is natural to start with the \emph{displacement energy}, which is a particular case of an external symplectic capacity of a subset of the symplectic space $V\times V^*$ (see \cite{hofer1994} for the details). The definition of the displacement energy operates with a time dependent compactly supported Hamiltonian $H(x, y, t)$ on $V\times V^*\times [T_1, T_2]$, whose \emph{total oscillation} is defined to be
\[
\|H\| = \int_{T_1}^{T_2} \sup_{x,y} H(x, y, t) - \inf_{x, y} H(x, y, t)\; dt,
\]
see~\cite[Ch.~5]{hofer1994} or~\cite{polt2001} for further details. Sometimes the segment $[T_1, T_2]$ is normalized to be $[0,1]$, but actually this and the following definitions do not depend on this normalization. It is always possible to scale the time segment by $\alpha$ and multiply the Hamiltonian by $1/\alpha$ without changing the result of the Hamiltonian flow.

The \emph{displacement energy} of $X\subset V\times V^*$, denoted by $e(K\times B^\circ)$, is the infimum of $\|H\|$ over all compactly supported time dependent Hamiltonians such that the corresponding time dependent Hamiltonian flow $\phi_t$ takes $X$ off itself, that is 
\[
X \cap \phi_{T_2}(X) = \emptyset. 
\]

Now we are ready to state the first version of the conjectured property of symplectic capacities that is related to the Bang problem:

\begin{conjecture}
\label{conjecture:plank-displacement}
If a convex body $K\subset V$ can be covered with a finite set of planks with the sum of widths (measured relative to the unit ball $B$) equal to $w$ then $e(K\times B^\circ)\le 2w$.
\end{conjecture}

Now we use a simple argument to establish a particular case of Conjecture \ref{conjecture:plank-displacement}. For every plank $P_i$ in question, we may choose either of the unit normals $\pm n_i\in V^*$, where by a \emph{unit normal} to a plank $P_i$ we mean an element $n_i\in V^*$ such that $\|n_i\|_*=1$ and the defining hyperplanes of $P_i$ are defined by $\{x\in V : \langle n_i, x\rangle = \const\}$.

\begin{theorem}
\label{theorem:almost-par}
Conjecture \ref{conjecture:plank-displacement} holds if the unit normals $n_i\in V^*$ of the planks can be chosen so that, for any sequence of non-negative coefficients $c_i$, at least one of which is $1$, the following inequality holds:
\begin{equation}
\label{equation:almost-par}
\left\|\sum_{i} c_i n_i \right\|_*\ge 1.
\end{equation}
\end{theorem}

Let us call the assumption \emph{almost parallel planks}. In the Euclidean case this assumption is guaranteed by a simpler assumption that $n_i\cdot n_j \ge 0$ for any pair of indices. This rarely happens for a random set of normals.

\begin{proof}
Let a plank $P_i$ have width $w_i$. Consider the Hamiltonian $H_i(x, y, t)$ defined for $t\in [2i-2, 2i]$ so that $H_i$ is independent of $t$ and $y$, equals zero for $x$ on one side of $P$, equals $w_i$ for $x$ on the other side of $P_i$, and changes linearly in $x$ inside $P_i$. The effect of the corresponding (discontinuous!) flow is as follows: $(x,y)\in V\times V^*$ remains fixed if $x$ is outside $P_i$ and gets shifted by $(0, 2n_i)$, where $n_i = d_x H_i$ is the unit normal to $P_i$, for $x$ in the interior of $P_i$. For $x\in\partial P_i$ we have a discontinuity. From the definition it follows that the part $(K\cap P_i)\times B^\circ$ gets shifted outside $K\times B^\circ$, spending the total oscillation $2w_i$. 

The idea is to shift this way everything outside $K\times B^\circ$ in a finite sequence of such steps for all planks $P_i$. The total oscillation of such a sequence of Hamiltonians is therefore twice the sum of widths. To make this idea work we need some care. First, the function $H_i(x, y, t)$ is not smooth for $x\in \partial P_i$ and therefore the Hamiltonian flow is discontinuous. This could be remedied by a certain smoothening changing the value of $d H$ in a small neighborhood of $\partial P_i\times V^*$; but after that we have to keep in mind that some parts near boundaries of planks are ``incompletely shifted'' by vectors $(0, 2 c_i(x) n_i)$ for some $0\le c_i(x)\le 1$.

A more serious problem, is that we have made several shifts and it may happen that something, previously shifted outside $K\times B^\circ$, returns inside $K\times B^\circ$ on a subsequent shift. The assumption of the theorem means that the total shift of the $V^*$ component over a point $x\in V$, which is equal to
\[
\sum_i 2 c_i(x) n_i,
\]
always has $\|\cdot\|_*$-norm at least $2$ and therefore indeed shifts the considered point $(x,y)$ outside $K\times B^\circ$.
\end{proof}

\begin{remark}
Here we observe a strange phenomenon. The classical method of Bang works better when the planks are far from parallel, see~\cite{ball1992} for an impressive example. But the displacement energy approach presented above likes the opposite situation, when the planks are almost parallel.
\end{remark}

\begin{remark}
Yaron~Ostrover has noted in private communication that the proof of the above theorem does not use the convexity of $K$. This may be useful, though lower bounds for the symplectic invariants of $K\times B^\circ$ seem less accessible for non-convex $K$.
\end{remark}

It is well known that the Hofer--Zehnder symplectic capacity $c_{HZ}(U)$ (see the definition and discussion in~\cite{hofer1994}) gives a lower bound for the displacement energy of $U$, where $U$ is an open bounded set in $V\times V^*$. Therefore it makes sense to consider a version of Conjecture \ref{conjecture:plank-displacement} for the Hofer--Zehnder capacity in place of the displacement energy.

\begin{conjecture}
\label{conjecture:plank-hz}
If $K$ can be covered with a finite set of planks with the sum of relative widths equal to $w$ then $c_{HZ}(K\times B^\circ)\le 2w$.
\end{conjecture}

This conjecture is weaker than Conjecture \ref{conjecture:plank-displacement} in view of the inequality $c_{HZ}(X)\le e(X)$ (see \cite{hofer1994}). If fact, for convex bodies $X\subset\mathbb R^{2n}$ there is still no evidence that the different symplectic capacities may have different values, and therefore there is no evidence that the versions of this conjecture with different capacities are really different. Moreover, in \cite[Remark 4.2]{aaok2013} it was shown that for centrally symmetric convex $K\subset V$ and $T\subset V^*$ the value $c_{HZ}(K\times T)$ always coincides with the displacement energy and the cylindrical capacity of $K\times T$ just because $K\times T$ can be put to a convex symplectic cylinder (see also Definition \ref{definition:convex-cylinder} below) of capacity $c_{HZ}(K\times T)$. Therefore in the case of centrally symmetric $K$ and $B$ these conjectures coincide.

\section{Billiards and capacity}
\label{section:billiards}

\subsection{Overview of known results on the symplectic approach to billiards}

If Conjecture~\ref{conjecture:plank-hz} (or a similar conjecture) holds, in order to produce Bang-type results we still need to calculate of estimate from below the capacity $c_{HZ}(K\times B^\circ)$. Fortunately, Shiri Artstein-Avidan and Yaron Ostrover established~\cite{aao2012} a nice elementary description of this capacity (all the bodies are assumed to be sufficiently smooth):

\begin{theorem}[Artstein-Avidan, Ostrover, 2012]
\label{theorem:aao-billiards}
The Hofer--Zehnder capacity $c_{HZ}(K\times B^\circ)$ is equal to the length of the shortest closed Minkowski billiard trajectory in $K$, where the length is measured in the norm $\|\cdot\|$ with unit ball $B$ and the reflection rule reflects the momentum coordinate from one point on $\partial B^\circ$ to the other point on $\partial B^\circ$ by combining it with a multiple of the normal to $\partial K$ at the hit point.
\end{theorem}  

\begin{remark}
In~\cite{aao2012} closed geodesics of $\partial K$ were also considered as a particular case of a billiard trajectory, with length measured with $\|\cdot \|$ norm. But in~\cite{abksh2014} it was shown that such closed trajectories can never be shorter than the ordinary piece-wise linear trajectories that reflect at the boundary of $K$ a finite number (in fact not exceeding $\dim K+1$) of times.
\end{remark}

Conjecture~\ref{conjecture:plank-hz} and Theorem \ref{theorem:aao-billiards} imply the following claim: If for a smooth strictly convex body $K$ the shortest closed Minkowski billiard trajectories in $K$ (with $\|\cdot\|$-length) have length $L$ then any system of planks covering $K$ has the sum of $\|\cdot\|$-widths at least $L/2$. In the Bang conjecture we consider the norm with unit ball $K-K$. In particular, for the original version of the Bang conjecture we need $L=2$ and no less. Unfortunately, the following example shows that we cannot guarantee $L=2$ already in the plane.

\begin{example}
\label{example:triangle}
If $K$ is the triangle in the plane (from the affine invariance we may assume $K$ regular) and the norm is defined by $K-K$, then the small triangle formed by its midpoints of sides can be verified to be a closed Minkowski billiard trajectory and have relative length $3/2$. The triangle is not smooth, but it can be smoothened without increasing the number $3/2$ too much. Thus the billiard approach together with Conjecture \ref{conjecture:plank-hz} is \emph{not sufficient} to establish the Bang conjecture already in this simple case.

Moreover, for the Euclidean norm, the regular triangle of unit width has a billiard trajectory along the midpoints of length $\sqrt{3}$. So the billiard approach fails even for the known case of the Bang theorem, which estimates the sum of Euclidean widths of the covering planks by the Euclidean \emph{minimal width} of $K$.
\end{example}

This shows that the symplectic method is not directly applicable to the still open non-symmetric case of the Bang conjecture (although Theorem \ref{theorem:billiards-symmetrization} below shows it can still produce a very good estimate). On the positive side, for the already known symmetric case the required bound was established in~\cite{aaok2013}:

\begin{theorem}[Artstein-Avidan, Karasev, Ostrover, 2013]
\label{theorem:billiards-in-ball}
Let $\|\cdot\|$ be a smooth norm and let its dual $\|\cdot\|_*$ be also smooth. Then any closed billiard trajectory in the unit ball $B$, being measured with $\|\cdot\|$, has length at least $4$ and $c_{HZ}(B\times B^\circ) = e(B\times B^\circ) = 4$.
\end{theorem}

In this theorem, the segment $[x, -x]\in B$, where we take any $x\in \partial B$, passed forth and back is a closed billiard trajectory of length $4$ in the norm associated with $B$. Theorem~\ref{theorem:billiards-in-ball} asserts that this is the shortest one, and together with Conjecture~\ref{conjecture:plank-hz} implies Ball's theorem from~\cite{ball1991} about the Bang problem in the centrally symmetric case. Clearly, this result together with Theorem~\ref{theorem:almost-par} already gives a symplectic proof for the particular case of Ball's theorem, when the ``almost parallel planks'' assumption (\ref{equation:almost-par}) is satisfied.

For possibly non-symmetric convex bodies a similar result was established in~\cite{abksh2014} (in this theorem we allow a norm to violate the reflexivity property $\|x\| = \|-x\|$):

\begin{theorem}[Akopyan, Balitskiy, Karasev, Sharipova, 2014]
\label{theorem:billiards-non-reflexive}
Let $\|\cdot\|$ be a smooth non-symmetric norm in $\mathbb R^n$ and let its dual $\|\cdot\|_*$ be also smooth. Then any closed billiard trajectory in the unit ball $K$, measured with $\|\cdot\|$, has length at least $2+2/n$.
\end{theorem}

\subsection{Billiard estimate for the Lagrangian product in the Bang problem}

Now we prove one more estimate related to the non-symmetric case of Bang's problem. It resembles Theorem~\ref{theorem:billiards-non-reflexive}, but we give a separate proof and do not see if one of the results follows from another.

\begin{theorem}
\label{theorem:billiards-symmetrization}
Let $K$ be a smooth strictly convex body in $\mathbb R^n$. Consider the norm with the unit ball $B=K-K$, then any closed billiard trajectory in $K$ with this norm has length at least $1+\frac{1}{n}$.
\end{theorem}

We postpone the proof and discuss the result first. This estimate is obviously tight for $n=1,2$, and is actually tight for $n\ge 3$, as it was checked by Yoav~Nir~\cite[Ch.~4]{nir2013}. In fact, a closed polygonal line with vertices at all the centers of mass of facets of any simplex $K$ is a closed billiard trajectory in $K$ with respect to the norm with unit ball $K-K$.

Assuming Conjecture~\ref{conjecture:plank-hz} (or a similar conjecture), this theorem would imply a weaker result than the Bang conjecture, that is the sum of relative widths of planks would be proved to be at least $\frac{n+1}{2n}$. This is not what was conjectured by Bang, but would be a good step towards the Bang conjecture. Evidently, for almost parallel planks we have the following corollary:

\begin{corollary}
\label{corollary:weak-parallel-bang}
Assume a convex body $K\subset V$, $\dim V = n$, is covered by planks $P_1,\ldots, P_N$. Let us measure the widths with the norm $\|\cdot\|$ with unit ball $K-K$, and assume that the unit (with respect to the dual norm) normals to the planks, $n_1,\ldots, n_N\in V^*$ can be chosen ``almost parallel'', that is for any sequence of non-negative coefficients $c_i$, at least one of which is $1$, the following inequality holds:
\begin{equation}
\left\|\sum_{i} c_i n_i \right\|_*\ge 1.
\end{equation}
Then the sum of widths of the planks ca be estimated as
\[
\sum_i w(P_i) \ge \frac{n+1}{2n}.
\]
\end{corollary}

Now we go down to the proof of Theorem \ref{theorem:billiards-symmetrization}. We need the following lemma to prove this theorem and Theorem~\ref{theorem:diff-general} below. 

\begin{lemma}
\label{lemma:covering-connected}
Let $K\subset\mathbb R^n$ be a convex body and $\|\cdot\|$ be the norm with unit ball $K-K$. If $C\in\mathbb R^n$ is a connected graph with total $\|\cdot\|$-length at most $h$, then $C$ can be covered by a translate of the homothet $h K$.
\end{lemma}

\begin{remark}
If we consider arbitrary centrally symmetric norm $\|\cdot\|_B$ with unit ball $B$, not connected to $K-K$, then Lemma \ref{lemma:covering-connected} holds true with the modified assumption: $\|\cdot\|_B$-length of the graph must be at most $h w_{B}(K)$, where $w_{B}(K)$ is the minimal $\|\cdot\|_B$-width of $K$. This generalization evidently follows from the inequality $w_B(K) \|\cdot\|_{K-K} \le \|\cdot\|_B $.
\end{remark}

\begin{proof}[Proof of Lemma \ref{lemma:covering-connected}]
We may assume that $C$ has straight line segments as edges. For an edge $[a,b]$ the inequality 
\[
\|a - b\| \le \delta
\]
in the norm with unit ball $K-K$ is equivalent to saying that $[a,b]$ can be covered with a translate of $\delta K$. So we cover all edges of $C$ (that is the whole $C$) by translates $\delta_1K+t_1, \ldots, \delta_mK+t_m$ with 
\[
\delta_1 + \delta_2 + \dots + \delta_m \le h.
\]

Then we observe that if two sets $\delta_iK+t_i$ and $\delta_j K + t_j$ intersect then they can be covered by a single set $(\delta_i+\delta_j)K + t'$. Indeed, we may consider $K$ smooth and strictly convex, having in mind an approximation argument. Then the smallest homothet $K_{ij} = hK + t'$ containing $K_i = \delta_i K + t_i$ and $K_j= \delta_j K + t_j$ must have a common supporting hyperplane with $K_i$ at a point $p_i\in \partial K_{ij}\cap\partial K_i$ and must have a common supporting hyperplane with $K_j$ at a point $p_j\in \partial K_{ij}\cap\partial K_j$. Moreover, from the minimality of this homothet $K_{ij}$ we may conclude that the supporting hyperplanes to $K_{ij}$ at $p_i$ and $p_j$ are parallel, see Figure~\ref{fig:two-body}. 
Hence the segment $p_ip_j$ is covered by a translate of $hK$, but if we make this segment slightly longer then it will no more be covered by a translate of $hK$; from the definition of the norm $\|\cdot\|$ this means $\|p_j - p_i\| = h$. But if $p_{ij}\in K_i\cap K_j$, then from the same definition
\[
\|p_i - p_{ij}\|\le \delta_i,\quad \|p_j - p_{ij}\| \le \delta_j \Rightarrow h = \|p_j - p_i\| \le \delta_i + \delta_j
\]
from the triangle inequality.

Using the connectedness of $C$ we can repeat this step several times to cover the whole $C$ with a translate of $(\delta_1+\dots + \delta_m)K$.
\end{proof}

\begin{center}
\includegraphics{fig-symplectic-bang-3.mps}\\
\f \label{fig:two-body} 
\end{center}

\begin{proof}[Proof of Theorem~\ref{theorem:billiards-symmetrization}]
By~\cite[Theorem~2.1]{abksh2014} the shortest closed billiard trajectory in $K$ has at most $n+1$ bounce points $\{q_i\}_{i=1}^m$ and cannot be covered by a smaller positive homothet of $K$. Applying Lemma~\ref{lemma:covering-connected} (explained later) to the closed trajectory with one segment removed we have:
\[
\sum_{i=2}^m \|q_i - q_{i-1}\| \ge 1,
\]
If $L$ is the $\|\cdot\|$-length of the closed polygonal line $q_1,q_2,\ldots, q_m,q_1$ then the above inequality is a lower bound for $L$ minus the length of the segment $[q_m, q_1]$. The same argument applies to any other segment, and since at least one of them has length at least $\frac{L}{n+1}$ (remember that $m\le n+1$) then 
\[
\left(1 - \frac{1}{n+1}\right) L \ge 1,
\]
that is $L \ge \frac{n+1}{n}$.
\end{proof}

\begin{remark}
Following~\cite{aaok2013} and assuming a version of Claude Viterbo's conjecture~\cite{vit2000} (volume of a convex $X\subset \mathbb R^{2n}$ is at least $\frac{c_{HZ}(X)^n}{n!}$), this theorem would also imply a Mahler-type inequality:
\begin{equation}
\label{equation:bad-viterbo}
\vol K \cdot \vol (K-K)^\circ \ge \frac{\left(1+\frac{1}{n}\right)^n}{n!}.
\end{equation}
The unknown referee has made the following observation about this inequality. Combining the classical Rogers--Shephard inequality
\[
\vol K \ge \frac{1}{\binom{2n}{n}} \vol (K-K)
\]
with the conjectured Mahler inequality
\[
\vol (K-K) \cdot \vol (K-K)^\circ \ge \frac{4^n}{n!}
\]
we obtain
\[
\vol K \cdot \vol (K-K)^\circ \ge \frac{4^n}{n! \binom{2n}{n}},
\]
which is of order $\frac{\sqrt{n}}{n!}$ and is better than \eqref{equation:bad-viterbo}. A particular conclusion is that $K\times (K-K)^\circ$ is far from symplectic balls or other convex bodies that satisfy the Viterbo inequality.
\end{remark}

%%%Unrelated stuff
%As a more elementary example of this activity, we want to mention another result, implicit in~\cite{bb2009}:

%\begin{theorem}[D.~Bezdek, K.~Bezdek., 2009]
%\label{theorem:planar-bezdek}
%For a convex body $K$ of constant width $1$ in the plane with the Euclidean norm, any shortest closed billiard trajectory has length $2$ and must be a diameter of $K$ passed twice.
%\end{theorem}

%As was noted by Alexey~Balitskiy (private communication), the proof of~\cite[Theorem~1.2]{bb2009} proves this assertion as well. Example~\ref{example:triangle} shows that it cannot be generalized to arbitrary norm without additional assumptions. Also, the generalization of Theorem~\ref{theorem:planar-bezdek} for the Euclidean norm in dimensions more than $2$ is open.

\section{Covering by symplectic cylinders}
\label{section:cylinders}

In~\cite{ball2001} Ball established another result similar to the Bang problem: When the unit ball in $\mathbb C^n$ is covered by unitary cylinders $Z_i$ of radii $r_i$ then $\sum_i r_i^2\ge 1$. Here $\mathbb C^n$ is endowed with the standard Hermitian form
\[
h(u, v) = \sum_{i=1}^n u_i \overline{v_i}
\]
and the corresponding norm $\|u\| = \sqrt{h(u,u)}$. The unit ball is considered in this metric, and a unitary cylinder of radius $r$ is an $r$-neighborhood (in the Hermitian norm) of a complex hyperplane.

This result seems even more suitable for the application of symplectic methods, because $\mathbb C^n$ itself has the symplectic structure $\Im h$ and the unitary cylinders $Z_i$ are symplectic cylinders with the capacities $c_{HZ}(Z_i) = e(Z_i)=\pi r_i^2$. So we are forced to state a more general conjecture, that would also imply Conjecture~\ref{conjecture:plank-hz}.

\begin{definition}
\label{definition:convex-cylinder}
For a convex body $S\subset \mathbb C$ call $Z = S\times \mathbb C^{n-1}\subset \mathbb C^n$, and all its images under linear symplectic transformations plus translations, a \emph{convex symplectic cylinder} with cross-section $S$.
\end{definition}

It is relatively clear that, for such linear and even more general non-linear images of such cylinders, the invariants $c_{HZ}(Z) = e(Z)$ equal the area of $S$. In general, the term ``symplectic cylinder'' means arbitrary symplectomorphic image of a standard cylinder, but here we are only interested in the linear images of convex cylinders. The corresponding version of our conjecture now becomes:

\begin{conjecture}
\label{conjecture:cylinders}
If a convex body $K\subset \mathbb C^n$ is covered by a finite set of convex symplectic cylinders $\{Z_i\}$ then 
\[
c_{HZ}(K) \le \sum_i c_{HZ}(Z_i).
\]
\end{conjecture}

Example \ref{example:curved-cylinders} and other examples in Section \ref{section:examples} show that the convexity assumption for the cylinders is crucial in this conjecture. We state two obvious by now lemmas:

\begin{lemma}
Ball's complex plank theorem follows from Conjecture \ref{conjecture:cylinders}.
\end{lemma}

\begin{proof} 
The unit ball $B\subset\mathbb C^n$ has capacity $c_{HZ}(B) = \pi$ and a unitary cylinder of radius $r$ is a particular case of a convex symplectic cylinder with all capacities equal to $\pi r^2$.
\end{proof}

\begin{lemma}
Conjecture~\ref{conjecture:plank-hz} follows from Conjecture \ref{conjecture:cylinders}.
\end{lemma}

\begin{proof}
For a convex body $K\subset\mathbb R^n$ (not in $\mathbb C^n$) and every plank $P_i$ of its covering, the convex body $P_i\times B^\circ\subset V\times V^*$ can be covered by a symplectic cylinder of capacity $2w(P_i)$. Indeed if $n_i\in V^*$ is the unit normal of $P_i$ (with $\|n_i\|_* = 1$) and $P_i$ is given by the inequality $\{a \le \langle n_i ,x \rangle \le b\}$, then $w(P_i) = b - a$. Also take a vector $v_i$ such that $\|v_i \| = 1$ and $\langle n_i ,v_i \rangle = 1$. Then the set
\[
Z_i = \left\{(x,y)\in V\times V^* : x \in P_i,\ \left| \langle y, v_i\rangle \right|  \le 1 \right\}
\]
is a convex symplectic cylinder ($n_i$ and $v_i$ produce a pair of symplectic canonical coordinates) of capacity $2w(P_i)$.

Since $K$ is covered by the $P_i$, the product $K\times B^\circ$ is covered by the $P_i\times B^\circ$, and is therefore covered by the convex symplectic cylinders $Z_i$. It remains to apply Conjecture~\ref{conjecture:cylinders}.
\end{proof}

Conjecture \ref{conjecture:cylinders} looks like a \emph{subadditivity property} of the Hofer--Zehnder capacity: The capacity of the union is at most the sum of capacities. In the next section we collect negative and positive evidence on the subadditivity property of symplectic capacities.

\section{Evidence on the subadditivity of capacities}
\label{section:subadd}

\subsection{Examples when subadditivity fails}
\label{section:examples}

In this section we discuss possible subadditivity properties of a symplectic capacity in more detail. We provide examples, which would show the absence of subadditivity of capacities when the convexity assumption is dropped. There exist a range of capacities (see for example~\cite{fgs2004}) $c(X)$ between the Hofer--Zehnder capacity $c_{HZ}(X)$ and the displacement energy $e(X)$, our examples actually apply to any such capacity.

\begin{example}
The first example is very simple. Let $B$ be a unit disc in the plane an let $S_-$ and $S_+$ be the halves of its boundary. If we thicken $S_-$ and $S_+$ slightly, their capacities still remain very close to zero. But their union has at least the same capacity as $B$ itself, which is $\pi$.
\end{example}

\begin{example}
In the notation of Section~\ref{section:symplectic}, let $B$ be the unit ball in the Euclidean space $\mathbb R^n$, the polar unit ball will be identified with $B$. Then the set $X = (B\setminus rB)\times B$ has displacement energy at most $2(1-r)$, the Hamiltonian
\[
H(x) = \left\{
\begin{array}{ll}
1-r,& \|x\| \le r\\
1-\|x\|,& r\le \|x\| \le 1\\
0,& \|x\|\ge 1
\end{array}
\right.
\]
does the job after certain smoothening, because its gradient has norm is at least $1$ over the base $B\setminus rB$ of $X$, which is sufficient to displace the ball bundle over this base constituting the set $X$. Hence the capacity if this set is also at most $2-2r$. The other part $Y = rB\times B$ obviously has displacement energy and any capacity at most $4r$. 

In total we have at most $2+2r$ for $c(X)+c(Y)$, but the union set $X\cup Y = B\times B$ has $c(B\times B)\ge 4$, according to Theorem~\ref{theorem:aao-billiards}. Here the set $X$ was topologically nontrivial, but we can easily remove a cylinder of radius $r$, passing from the origin to the boundary of $B$ from $X$, and add this cylinder to $Y$, without increasing the capacity of $Y$ by Theorem~\ref{theorem:aao-billiards}. Moreover, after removing the cylinder it is possible to slightly modify $X$ and $Y$ so that they both, together with their intersection, become starshaped. In this example the sets and their intersection cannot be distinguished from convex sets from the topological viewpoint.
\end{example} 

\begin{example}
\label{example:curved-cylinders}
We describe the construction communicated by Leonid~Polterovich (a version of which appeared in \cite{rudyakschlenk2007}) showing that the subadditivity fails when we cover any set by not necessarily convex symplectic cylinders, that is non-linear symplectomorphic images of symplectic cylinders. First observe that any bounded subset $K\subset \mathbb C^n$ can be covered by a cubic grid with diameter of cubes at most $\varepsilon>0$. Then we can partition all the cubes of the grid into $2^{2n}$ disjoint families of cubes, ``colors'', sorting them by the parity vectors of their coordinates. A more careful procedure with a modified grid \cite{rudyakschlenk2007} allows to use $2n+1$ colors, but here it is not relevant.

After that we consider a color of disjoint small cubes and produce a Hamiltonian symplectomorphism $\phi : \mathbb C^n \to \mathbb C^n$ that keeps the shape of all these cubes, but arranges their centers along a given straight line. Indeed, we can  continuously move (in a certain order of the cubes) a small cube $C_i$ to its desired position $C_i'$ on the line so that it does not come more than $\delta$ close to any other cube of the same color during this movement, for some positive $\delta < \epsilon$. This motion corresponds to a Hamiltonian motion of the whole $\mathbb C^n$ (with a time-dependent linear on $\mathbb C^n$ Hamiltonian), and it is possible to modify this time dependent Hamiltonian so that it remains the same on the moving cube and becomes zero outside the $\delta$-neighborhood of the moving cube. Thus modified Hamiltonian symplectomorphism moves one cube to its desired place not touching the other small cubes of the same color; composing several such symplectomorphisms we arrange all the cubes of a given color in a line and easily cover them by a symplectic cylinder of capacity at most $\epsilon^2$.

Looking at the situation the other way, we cover the original color of disjoint cubes by the inverse symplectomorphic image of the final cylinder of capacity $\epsilon^2$. Applying the above observation to every one of the $2^{2n}$ colors of small disjoint cubes, we spend the total capacity at most $2^{2n} \varepsilon^2$ to cover them with cylinders. This example shows that in the case of covering by cylinders their convexity must be essential.
\end{example}

The given examples show that the subadditivity seems to strongly depend on convexity, thus we only restrict ourselves to convex sets in Conjecture \ref{conjecture:subadditivity}.

\subsection{Cutting the Euclidean ball into two convex pieces}

In~\cite[Theorem~2.2]{zehm2012} an opposite inequality for the Hofer--Zehnder capacity was proved in a particular case using pseudoholomorphic curves: If two disjoint convex bodies $K_1$ and $K_2$ are contained in the Euclidean ball $B\subset\mathbb C^n$ then 
\[
c_{HZ}(K_1) + c_{HZ}(K_2) \le c_{HZ}(B).
\]
In view of the monotonicity of capacities and the hyperplane separation of convex bodies by the Hahn--Banach theorem, this inequality is equivalent to its particular case when the two convex bodies are produced by a hyperplane cut of the ball. For a hyperplane cut of the ball, the validity of Conjecture~\ref{conjecture:subadditivity} would thus imply equality; we prove this equality directly, using that the characteristics on the boundary of the ball are relatively easy to understand.

\begin{lemma}
If a ball $B\subset\mathbb R^{2n}$ is cut by a hyperplane into pieces $K_1$ and $K_2$ then 
\[
c_{HZ}(K_1) + c_{HZ}(K_2) = c_{HZ}(B).
\]
\end{lemma}

\begin{remark}
After the preprints of this text appeared, Pazit Haim-Kislev proved in \cite{haimkislev2017}, using Frank H. Clarke's approach to closed characteristics from \cite{clarke1979}), that the subadditivity actually holds for any hyperplane cut of any convex body $K\subset\mathbb C^n$, thus providing much stronger evidence of the subadditivity property.
\end{remark}

\begin{proof}
Assume that the radius of $B$ is $1$ and identify $\mathbb R^{2n} = \mathbb C^n$. Using the transitivity of the $\mathrm{U}(n)$ action on $\partial B$ assume that the cutting hyperplane is $\Pi = \{\Re z_1 = \cos \tau_0\}$.

Let us use the description of the Hofer--Zehnder capacity of convex bodies in terms of the minimal action of a closed characteristic on the boundary. Obviously, the closed characteristic $(e^{it}, 0, \ldots, 0)\subset \partial B$ is broken by $\Pi$ into two closed characteristics of $K_1$ and $K_2$ respectively, and their actions sum up to $c_{HZ}(B) = \pi$. This establishes 
\[
c_{HZ} (K_1) + c_{HZ} (K_2) \le c_{HZ} (B) = \pi, 
\]
moreover, this also proves that $K_1$ and $K_2$ can be covered by convex symplectic cylinders based on the two-dimensional sections of $K_1$ and $K_2$ with sum of the capacities equal to $\pi$.

It remains to show that other closed characteristics on $\partial K_1$ or $\partial K_2$ have larger actions.

Assume a closed characteristic for $K_1 = \{\Re z_1 \ge \cos \tau_0\}$ starts at a point $(z_1, \ldots, z_n)$ such that $\Re z_1 = \cos \tau_0$ and $\Im z_1 < 0$. Put $z_1 = \rho e^{-i\tau}$. We must have $\rho\le 1$ and if $\rho=1$ then this closed characteristic is the one already considered. So we assume $\rho < 1$.

Let us check how this point evolves along a characteristic on $\partial K_1$. First, it moves along $\partial B$ as $(\rho e^{i(-\tau + t)}, z_2e^{it}, \ldots, z_ne^{it})$ for $t\in[0, 2\tau]$. Then it moves in $\Pi$ along the direction of $\Im z_1$ getting from $(\rho e^{i\tau}, z_2e^{2i\tau}, \ldots, z_ne^{2i\tau})$ to $(\rho e^{-i\tau}, z_2e^{2i\tau}, \ldots, z_ne^{2i\tau})$. Then everything is repeated. In order for this point to get to its original position we must have 
\[
\tau = \pi\frac{k}{m}.
\]
Then the total number of turns will be $m$ and the action will be:
\[
A = k \left( \rho^2 (\tau - \sin\tau \cos\tau) + \pi (1 - \rho^2) \right),
\]
What remains to show is the inequality:
\[
\rho^2 (\tau - \sin\tau \cos\tau) + \pi (1 - \rho^2) \ge \tau_0 - \sin\tau_0 \cos \tau_0.
\]
Note that $x-\sin x\cos x$ increases on $[0, \pi]$ from $0$ to $\pi$. 
In the case $\tau, \tau_1\ge \pi/2$ we have
\[
\rho^2 (\tau - \sin\tau \cos\tau) + \pi (1 - \rho^2) \ge \tau - \sin\tau \cos\tau \ge \tau_0 - \sin\tau_0 \cos \tau_0,
\]
because from $\rho \cos \tau = \cos \tau_0$ it follows that $\tau \ge \tau_0$ when they both are greater than $\pi/2$. 

In the remaining case $\tau \le \tau_0 \le \pi/2$ we substitute $\rho \cos \tau = \cos \tau_0$ and we have to prove the inequality
\[
\cos^2\tau_0 (\tau - \sin\tau \cos\tau - \pi) + \pi \ge \cos^2\tau (\tau_0 - \sin\tau_0 \cos \tau_0).
\]
In the considered range the left hand side is increasing and the right hand side is decreasing in $\tau$, hence it remains to consider the case $\tau = 0$, when the inequality is 
\[
\pi - \pi \cos^2 \tau_0 \ge \tau_0 - \sin\tau_0 \cos \tau_0.
\] 
Putting $x = 2\tau_0\in (0, \pi)$ and using the trigonometric identities we have to prove 
\[
\pi(1 - \cos x) \ge x - \sin x,
\]
the latter is true since for $x\ge \pi/2$ the left hand side is at least $\pi$ and the right hand side is at most $\pi$, while for $x\in [0, \pi/2]$ the inequality is true since it holds for $x=0$ and after taking the derivative is becomes $\pi\sin x\ge 1 - \cos x$, which is obviously true for $x\in [0, \pi/2]$.
\end{proof}

\subsection{Decomposing Hamiltonian symplectomorphisms}

Let us just mention one sort of subadditivity that exists near the notion of a symplectic capacity. One way is to define capacities through action selectors that choose an action of a fixed point of a compactly supported Hamiltonian symplectomorphism $\phi : \mathbb R^{2n} \to \mathbb R^{2n}$ through certain topological constructions. In particular, Viterbo in~\cite{vit1992} uses the generalized generating functions for Hamiltonian symplectomorphisms and defines an action selector $c_+(\phi)$ for compactly supported Hamiltonians $\phi :\mathbb R^{2n}\to \mathbb R^{2n}$ satisfying
\[
c_+(\phi\psi) \le c_+(\phi) + c_+(\psi).
\]
This action selector gives rise to a symplectic capacity $c_V(U)$ for open bounded subsets $U\subset \mathbb R^{2n}$, as defined in~\cite[Definition~4.11]{vit1992} by 
\[
c_V(U) = \sup\{c_+(\phi) : \supp \phi\subset U\}.
\]

Now, in order to have a subadditivity for $c_V$ it were sufficient to have a claim like this: For two bounded open subsets $U, V\subset \mathbb R^{2n}$ and a Hamiltonian symplectomorphism $\tau$ supported in $U\cup V$, there exist two Hamiltonian symplectomorphisms $\phi$ and $\psi$ with supports in $U$ and $V$ respectively such that $\tau = \phi\circ \psi$. \emph{Unfortunately}, this claim cannot be true, since any $\tau$ taking a point from $U\setminus V$ to $V\setminus U$ cannot be decomposed this way. Already in the plane we can consider $U$ and $V$ as unions of several disjoint squares obtained from $Q_0=(-\varepsilon, 1 + \varepsilon)\times (-\varepsilon, 1 + \varepsilon)$ by translations $(m, 0)$, with $m$ odd for $U$ and even for $V$. If $U\cup V$ is connected and consists of $2N$ copies of $Q_0$, then an appropriately chosen diffeomorphism of it evidently cannot be decomposed in less than $2N$ diffeomorphisms supported in either $U$ or $V$. Of course this latter example does not apply to connected sets.

If one wants to utilize this decomposition approach somehow, some extra properties like the convexity of $K$ in Conjecture~\ref{conjecture:subadditivity} must be used. For example, we might want to bound $c_V(K)$ (or $c_{HZ}(K)$) in the left hand side and note that this number is achieved for convex $K$ at very special time-independent Hamiltonians, which might turn out to be decomposable.

\section{Inequalities between the oscillation and the norm of the differential}
\label{section:oscillation}

In this section we consider another problem that resembles the Bang problem and allows similar approaches. First, we start with an elementary particular case (see also~\cite[P.~113]{clarke1989}):

\begin{theorem}
\label{theorem:diff-on-ball}
Let $F$ be a $C^1$-smooth function on the unit ball $B$ of a norm $\|\cdot\|$. Then
\[
\max_{x\in B} F(x) - \min_{x\in B} F(x) \ge 2 \min_{x\in B} \|dF(x)\|_*,
\]
where $\|\cdot\|_*$ is the corresponding dual norm.
\end{theorem}

\begin{proof}
By the standard approximation argument we assume the norms $\|\cdot\|$ and $\|\cdot\|_*$ to be infinitely smooth and strictly convex. We also assume $F$ to be infinitely smooth. For any $x$ consider the unique unit vector $y(x)$ such that 
\[
\langle dF(x), y(x) \rangle = \|dF(x)\|_*.
\]
Under the above assumptions this unit vector depends smoothly on $x$ and we can consider the differential equation:
\[
\dot x = y(x).
\]
We consider its solution with the initial condition $x(0) = 0$. Since this solution has the unit velocity it cannot get outside $B$ in a period of time less than $1$. By the extension of solutions theorem the solution $x(t)$ is defined for $t\in (-1, 1)$. Then we calculate
\[
\frac{d}{dt} F(x(t)) = \langle dF(x), \dot x\rangle = \langle dF(x), y(x)\rangle = \|dF(x)\|_*. 
\]
The value $F(x(t))$ increases with $t$ and if we put $m = \min_{x\in B} \|dF(x)\|_*$ then $\frac{d}{dt} F(x(t)) \ge m$, and therefore $F(x(t))$ oscillates by at least $2m$ on $(-1, 1)$.
\end{proof}

The symplectic approach allows to prove a similar estimate:

\begin{theorem}
\label{theorem:diff-billiards}
Let $F$ be a $C^1$-smooth function on a convex body $K$, and let us measure everything with a norm $\|\cdot\|$ whose unit ball is $B$. Then
\[
\max_{x\in K} F(x) - \min_{x\in K} F(x) \ge \frac{1}{2} e(K\times B^\circ) \cdot \min_{x\in K} \|dF(x)\|_*.
\]
\end{theorem}

\begin{proof}
Consider $F$ as a Hamiltonian on $K\times V^*$ (where $V\supset K$ is the ambient space) and observe that its gradient flow has velocity $dF(x)$ in the direction of $V^*$, hence it shifts $K\times B^\circ$ off itself in time $\frac{2}{\min_{x\in K} \|dF(x)\|_*}$. So the total displacement energy of $K$ satisfies
\[
e(K\times B^\circ) \le \frac{2}{\min_{x\in K} \|dF(x)\|_*}\left(\max_{x\in K} F(x) - \min_{x\in K} F(x)\right),
\]
which is equivalent to what we need to prove.
\end{proof}

It turns out that the following version of Theorem~\ref{theorem:diff-billiards} can be proved without any symplectic techniques:

\begin{theorem}
\label{theorem:diff-general}
Let $F$ be a $C^1$-smooth function on a convex body $K\subset \mathbb R^n$ and let $\|\cdot\|$ be the norm with unit ball $K-K$. Then
\[
\max_{x\in K} F(x) - \min_{x\in K} F(x) \ge \min_{x\in K} \|dF(x)\|_*,
\]
where $\|\cdot\|_*$ is the corresponding dual norm.
\end{theorem}

\begin{proof}
Let us assume everything smooth and even real-analytic, and consider again the trajectories of the normalized gradient (in the sense of the $K-K$ unit ball norm) vector field in a neighborhood of $K$. Observe that assigning a trajectory of this vector field to a point $x\in K$ gives a continuous map $\phi$ from $K$ to a topological space of all trajectories, which has covering dimension at most $n-1$. 

The non-symmetric version \cite[Theorem~6.2]{kar2012} of the theorem on the Alexandrov width from \cite{abr1972,sit1958}, see also~\cite[Proposition~1, pp.~84--85, and Theorem~1, p.~268]{tikh1976}) asserts that whenever a convex body $K$ of dimension $n$ is continuously mapped, $\phi : K\to X$, to a topological space of covering dimension at most $n-1$, there exists a connected subset $C\subseteq K$ that is mapped by $\phi$ to a single point and that cannot be covered by a smaller homothet of $K$. 

Thus obtained set $C$ (mapped to a single point by $\phi$) is in fact a curve segment of the intersection of a trajectory of the vector field with $K$. If the $\|\cdot\|$-length of the trajectory $C$ is at least $1$ then we are done by integrating over this trajectory as in the proof of Theorem \ref{theorem:diff-on-ball}. Otherwise Lemma~\ref{lemma:covering-connected} asserts that $C$ can be covered with a smaller homothet of $K$, which is a contradiction.

The above argument (essentially due to Abramov and Sitnikov) seems to have never been published in English as a whole; below we provided an expanded and relatively self-contained version of it. Assume that for arbitrary $\epsilon > 0$ we have a trajectory $\gamma$ of the gradient flow such that $\gamma \cap K$ cannot be covered by a translate of the homothet $(1-\epsilon)K$. Then by Lemma \ref{lemma:covering-connected} the length of the smallest curve segment $S$ of $\gamma$ containing $\gamma\cap K$ (the latter set may not be connected) is at least $1 - \epsilon$. Integrating over this curve segment $S$ we see the oscillation of $F$ at least $(1-\epsilon) \min_{x\in K} \|dF(x)\|_*$. The theorem holds true if we have such inequality for every $\epsilon > 0$.

So assuming the contrary we take some $\epsilon > 0$ such that for every trajectory of the gradient flow $\gamma$ its part $\gamma\cap K$ can be covered by a translated $(1-\epsilon) K$. Assume also that the origin is in the interior of $K$, this implies $hK\subset \inte h'K$ for $h<h'$.

Consider a trajectory of the gradient flow $\gamma$, we may assume the gradient flow is extended to a neighborhood of $K$ and so is $\gamma$. From the assumption we cover the set $\gamma\cap K$ by a translate of $(1 - \epsilon)K$. Moreover, we can take two parameters $t_0, t_1$ on the curve so that $\gamma(t_0), \gamma(t_1)\not\ni K$, $\gamma$ only gets into $K$ between $\gamma(t_0)$ and $\gamma(t_1)$, and the curve segment $\gamma[t_0, t_1]$ is still covered by a translate of $(1-\epsilon/2)K$, put $h = 1 - \epsilon/2$ for brevity. It is clear that other close to $\gamma$ trajectories $\gamma'$ still have the part $\gamma'\cap K$ covered by the same translate $h K + v$.

Now we observe that the space of all trajectories that we work with has covering dimension at most $n-1$. In the real-analytic case we may parametrize such trajectories by the first point they enter $K$ through $\partial K$, thus making a parametrization by a semianalytic subset of $\partial K$ (we assume $\partial K$ real-analytic) and using the nice structural properties of semianalytic sets, e.g. from \cite{hardt1975}.

Now we are going to construct a map $\psi : K\to\mathbb R^n$ that has mutually exclusive properties: Its image lies in at most $(n-1)$-dimensional subset of $K$, and at the same time its image has nonempty interior. The definition of the covering dimension allows to produce a covering of the space of trajectories by open sets $U_i$ with multiplicity at most $n$ and such that for every $U_i$ there is a translate $h_i K + v_i$ that contains all $\gamma\cap K$ for any $\gamma\in U_i$. Make a partition of unity $\{\rho_i\}$ subordinated to $\{U_i\}$ and put
\[
\psi(x) = \sum_i \rho_i(\gamma_x) v_i,
\]
where $\gamma_x$ is the trajectory through $x$. From the covering property we always have 
\[
\| x - \psi(x)\| \le h < 1.
\]
Let us show that this implies that the convex body $(1-h)K$ is covered by the image of $\psi$. Indeed, for any point $y\in (1-h)K$ the map $\sigma : \partial K \to \partial K$ defined by
\[
\sigma(x) = \frac{\psi(x) - y}{\|\psi(x)-y\|}
\]
through the homotopy
\[
\frac{(1-t)\psi(x) + t x - y}{\|(1-t)\psi(x) + t x - y\|}
\]
becomes $x\mapsto \frac{x - y}{\|x - y\|}$. The homotopy is well defined because the inequality $\| x - \psi(x)\| + \|y\| < 1$ prevents the equality $x = (1-t) (x - \psi(x)) + y$ and keeps nonzero denominator. The additional homotopy
\[
\frac{x - (1-t)y}{\|x - (1-t)y\|}
\]
eventually takes $\sigma$ to the identity map of $\partial K$ of degree $1$. Hence $\sigma$ cannot be extended continuously to a map $K\to \partial K$ showing that for some $x\in K$ we must have $\psi(x) = y$.

Now the image of $\psi$ contains $(1-h)K$ and therefore has nonempty interior, while the formula of its definition and the $n$-fold covering assumption show that the image is lying in the union of countably many convex hulls of $n$-tuples of points in $\mathbb R^n$, $(n-1)$-dimensional simplices. This is a contradiction. 
\end{proof}

\begin{remark}
Theorem \ref{theorem:diff-general} is optimal because any linear function $F$ provides the equality case in Theorem \ref{theorem:diff-general}, since $B=K-K$ and the norm of a linear function $F$, as an element of $(\mathbb R^n)^*$, precisely equals its oscillation on $K$.

If we consider an arbitrary norm with centrally symmetric unit ball $B$ then for the corresponding norms of a vector $v$ we have
\[
\|v \|_B \ge w_B(K)\cdot \|v \|_{K-K}
\] 
and for a linear form $\lambda \in (\mathbb R^n)^*$ and its dual norms we have
\[
w_B(K)\cdot \|\lambda\|_{*, B} \le \|\lambda \|_{*,K-K}.
\]
Since the minimal width $w_B(K)$ of $K$ in the norm $B$ corresponds to a forth and back billiard trajectory, it follows that
\[
c_{HZ}(K\times B^\circ) \le 2 w_B(K)
\] 
and therefore Theorem \ref{theorem:diff-general} implies
\begin{multline*}
\frac{1}{2} c_{HZ}(K\times B^\circ) \cdot \min_{x\in K} \|dF(x)\|_{*, B} \le w_B(K) \cdot \min_{x\in K} \|dF(x)\|_{*, B} \le\\
\le  \min_{x\in K} \|dF(x)\|_{*, K-K} \le \max_{x\in K} F(x) - \min_{x\in K} F(x).
\end{multline*}
This shows that Theorem \ref{theorem:diff-general} is stronger than Theorem~\ref{theorem:diff-billiards} in the case of symmetric $B$ and $c_{HZ} (K\times B^\circ) = e(K \times B^\circ)$.
\end{remark}

\begin{remark}
It is curious that the proof of Theorem~\ref{theorem:diff-on-ball} works in infinite dimensional Banach spaces (for decent functions $f$), while the above argument to prove Theorem~\ref{theorem:diff-general} is essentially finite dimensional. Therefore its extension to infinite dimensional Banach spaces is an open problem.
\end{remark} 

%%% Moved to another paper
%Observe that by the way we have proved the following result resembling Gromov's waist of the sphere theorem~\cite{grom2003}:

%\begin{theorem}
%\label{theorem:waist-1dim}
%Assume a convex body $K\subset \mathbb R^n$ is mapped by a continuous map $f$ to a topological space $Y$ with covering dimension less than $n$. Then for some $y\in Y$ the total length of $f^{-1}(y)$ in the norm with unit ball $K-K$ is at least $1$.
%\end{theorem}

%\begin{proof} By the standard technique we may represent $f$ as a composition of two maps
%\[
%\begin{CD}
%X @>{\tilde f}>> \tilde Y @>{\pi}>> Y
%\end{CD}
%\]
%with $\pi$ having discrete fibers and $\tilde f$ having connected fibers. Then we apply~\cite[Theorem~6.2]{kar2012} and find $\tilde y\in \tilde Y$ such that $C=\tilde f^{-1}(\tilde y)$ (which is a connected component of a fiber of $f$) cannot be covered by a smaller homothet of $Y$. By Lemma~\ref{lemma:covering-connected} this $C$ has total length at least $1$ in the corresponding norm.
%\end{proof}

\section{Bang's problem for two directions of planks}
\label{section:two-dir}

In this section we prove a particular case of Bang's problem using elementary methods. It is independent of the symplectic considerations, but we thought it makes sense to confirm another particular case of the conjecture. One may check that it does not follow from the result under the ``almost parallel'' assumption of Theorem \ref{theorem:almost-par}.

\begin{theorem}
Let a convex body $K\subset \mathbb{R}^n$ be covered by a family of planks $P_1,\ldots, P_m$, whose normals have only two distinct directions. Then the sum of widths of the planks in the norm with the unit ball $K-K$ is at least $1$, that is the Bang conjecture holds in this case. 
\end{theorem}

\begin{proof} 
If all the planks are parallel to each other then the assertion is evidently true. Assume there are two distinct normals $n_1,n_2\in V^*$ (we put $V=\mathbb R^n$ and normalize $n_1$ and $n_2$ by the norm with unit ball $(K-K)^\circ$). Obviously, the projection 
\[
\pi : V \to \mathbb R^2,\quad \pi(x) = (n_1(x), n_2(x))
\]
reduces the problem to the following planar case: The projection (denote it by $K$ again) is inscribed in the unit square $abcd$ (let $a$ be the left bottom and $b$ be left top), that is $K$ contains points on every side of $abcd$. The unit square appears because the normalization of $n_1$ and $n_2$ simply means that their ranges on $K$ both have unit lengths.

Let the points where the projection of $K$ touches the sides $ab$ and $cd$ be $p$, $q$ respectively, see Figure~\ref{fig:abcd}. Assume $K$ to be covered by a set of horizontal and vertical planks with sum of widths (now the vertical and horizontal widths are in fact Euclidean) less than $1$. Also choose such a covering with the minimal number of planks.

If there are only two planks then the result is well known, see~\cite{ohmann1957} or \cite[Lemma~10.1.1]{bezdek2010book}. So we assume that there are $k$ vertical planks and at least $k$ horizontal planks (we interchange the axes if needed), $k>1$.

Consider the points of $K$ not covered with the vertical planks, they split into $k+1$ convex sets $M_1\cup M_2\cup\dots\cup M_{k+1}$ ordered from left to right, some of the $M_i$ may be empty. These sets have to be covered with horizontal planks and this reduces to cover their projection to the $0y$ axis with a set of segments. Definitely, one needs at most $k+1$ segments to cover those projections, and we know that $k$ segments are really needed. Now consider the cases (Figure~\ref{fig:cases} may be of help):

\begin{center}
	\parbox[b]{0.4\textwidth}{
		\begin{center}
			\includegraphics{fig-symplectic-bang-1.mps}\\
		\f \label{fig:abcd} 
		\end{center}
	}
	\hskip 0.3cm
	\parbox[b]{0.4\textwidth}{
		\begin{center}
			\includegraphics{fig-symplectic-bang-2.mps}\\
		\f \label{fig:cases}	
		\end{center}
	}
		
\end{center}

\begin{enumerate}
\item 
The set $M_1\ni p$ is nonempty and its projection to $0y$ has no intersection with the projections of other $M_i$'s. Then one horizontal plank is needed to cover $M_1$ separately from the other parts. But it makes sense to replace this plank with a vertical one, indeed, the set $M_1$ contains the triangle $pc_1d_1$ (see Figure~\ref{fig:abcd}) homothetic to $pcd$, whose vertical and horizontal widths coincide. Therefore the vertical width of $M_1$ is at least its horizontal width. So we replace the horizontal plank of $M_1$ with a vertical one and merge this vertical plank with the first vertical plank in the list. After that the sum of widths does not increase and the number of planks does decrease. 

\item
The case when the projection of the last $M_{k+1}\ni q$ to $0y$ does not intersect the other projections of $M_i$'s is considered similarly.

\item
The set $M_1$ is empty and $M_{k+1}$ is also empty. Then the projections of $M_i$'s to $0y$ can be covered with $k-1$ segments, but we have assumed that the number of segments is at least $k$.

\item 
$M_1=\emptyset$, $M_{k+1}$ is not empty and its projection to $0y$ intersects some of the projections of other $M_i$'s. Again, in this case at most $k-1$ horizontal planks are sufficient.

\item
Similar to the previous case, when we interchange $M_1$ and $M_{k+1}$. 

\item 
Both the projections of $M_1\ni p$ and $M_{k+1}\ni q$ to $0y$ are nonempty and both of them intersect other $M_i$'s. Again, we know that we really need at least $k$ horizontal planks to cover $M_i$'s. This may only happen when the projections of $M_1$ and $M_{k+1}$ do intersect and the projections of $M_i$'s with $2\le i\le k$ are disjoint from them and are disjoint from each other. 

Therefore there is a horizontal plank $P_h$ that covers both $M_1$ and $M_{k+1}$. 
Other sets $M_2,\dots,M_k$ then have to be disjoint from $P_h$ since the total number of needed horizontal planks is precisely $k$.
It is left to note, that the segment $[p, q]$ is covered by $P_h$, but $[p,q]$ should intersect all $M_i$, Figure \ref{fig:cases}, that leads to contradiction.
%
% Interchanging the vertical and horizontal direction we find a vertical stripe $P_v\ni p, r$. The segment $qs$ intersects the boundary of $P_v$ at $q_1,s_1$, lying in their respective $M_{i}$, $M_{i+1}$. But then $M_i$ and $M_{i+1}$ intersect $P_h$, which is a contradiction.

\end{enumerate}

\end{proof}

\section{Fractional Bang-type results}
\label{section:fractional}

\subsection{Linear programming considerations and covering by Euclidean cylinders}

A \emph{fractional Bang theorem} would be a result showing that if a convex body $K$ is covered by a family of planks so that every point of $K$ is covered at least $k$ times, then the sum of widths is at least $Wk$ for some constant $W$. Below we calculate $W$ for certain cases and show that it generally must be less than the constant from the original Bang-type results.

Another equivalent statement (explaining the term ``fractional'') would be to consider the planks with non-negative weights covering any point in $K$ with the sum of weights at this point at least $1$, and deduce that the weighed sum of widths of the planks is at least $W$. Again, we in principle measure the width of a plank in arbitrary norm. 

Minimization of the weighted sum of the planks is a linear programming problem with an infinite number of variables (weights of the planks) and constraints (the points with their requirements to be covered at least $1$ time). A lower bound $W$ in this problem is evidently given by any probability Borel measure $\mu$ in $K$ satisfying $w(P) \ge W \mu (P\cap K)$. A version of the Farkas lemma then concludes that the maximum of such lower bounds (the maximum is attained because of the compactness of the space of measures) is in fact equal to the minimum in the original problem, is we extend the original problem from finite collections of planks to integrals over a measure on the set of planks.

This technique is hard to apply for arbitrary bodies and norms, so we concentrate on the case of the Euclidean norm in $\mathbb R^n$ and its unit ball $B^n$. 

\begin{theorem}
If the unit Euclidean ball $B\subset\mathbb R^n$, for $n\ge 3$, is covered by a set of weighted planks $P_1,\ldots, P_N$ with respective weights $t_1,\ldots, t_N$ (so that every its point is covered with sum of weights at least $1$) then the weighted sum
\[
\sum_i t_i w(P_i) \ge W_n = \frac{\Gamma\left(\frac{n-1}{2}\right)\Gamma\left(\frac{1}{2}\right)}{\Gamma\left(\frac{n}{2}\right)}.
\]
\end{theorem}

\begin{proof}
For $n\ge 3$ we take $\mu$ to be a properly normalized surface area on $\partial B^n$. Its projection to a one-dimensional line will be a measure on the segment $[-1,1]$ with density proportional to $(1-x^2)^{\frac{n-3}{2}}$, for $n\ge 3$ this density has maximum $1$ at zero. When we normalize this measure to make it probability measure, there appear the factor $1/W_n$ at the projected density, where
\[
W_n = \int_{-1}^1 (1-x^2)^{\frac{n-3}{2}}\; dx = \frac{\Gamma\left(\frac{n-1}{2}\right)\Gamma\left(\frac{1}{2}\right)}{\Gamma\left(\frac{n}{2}\right)}.
\]
After such a normalization for any plank perpendicular to our projection we have
\[
w(P_i) \ge W_n \mu (P_i\cap B^n),
\]
and from the radial symmetry this applies to any plank in fact. Summation with weights then shows
\[
\sum_i t_i w(P_i) \ge W_n \sum_i t_i \mu (P_i\cap B^n) \ge W_n \mu (B^n) = W_n.
\]
\end{proof}

Then constant $W_n$ is of order $\frac{1}{\sqrt{\pi n}}$ for large $n$ (and equals $2$ for $n=3$ as expected). Actually, this constant cannot be improved, as the following argument shows. Take a sufficiently small $\delta>0$ and consider a set of $N$ random centrally symmetric planks of width $\delta$. Each such plank covers 
\[
\frac{\int_{-\delta/2}^{\delta/2} (1-x^2)^{\frac{n-3}{2}}\; dx}{\int_{-1}^1 (1-x^2)^{\frac{n-3}{2}}\; dx}
\]
of the surface area, which is close to $\delta / W_n$ for small $\delta$. Then the sum of widths is $\delta N$ and the expected covering multiplicity is approximately $\delta N / W_n$. Then some kind of central limit theorem shows that the minimal covering multiplicity can get sufficiently close to $\delta N/ W_n$ thus showing that $W_n$ is tight.

A similar argument with the uniform measure on $\partial B^n$ is applicable when we want to fractionally cover the Euclidean unit ball $B^n$ with $m$-dimensional Euclidean cylinders, that is sets congruent to $Z = X\times \mathbb R^m$, where $X$ is an $(n-m)$-dimensional convex body. We denote by $\sigma_{n-m}(Z) = \vol_{n-m} X$ the $(n-m)$-dimensional cross-section of $Z$. 

\begin{theorem}
Let $m\ge 2$. For a weighted covering of $B^n$ with $m$-dimensional Euclidean cylinders $Z_1, \ldots, Z_N$ with respective weights $t_1,\ldots, t_N$, we have:
\[
\sum_{i=1}^N t_i \sigma_{n-m}(Z_i) \ge \frac{n\pi^{n/2}}{\Gamma(n/2+1)} \frac{\Gamma(m/2)}{2\pi^{m/2}} = \frac{\pi^{\frac{n-m}{2}}\Gamma(m/2)}{\Gamma(n/2)}.
\]
\end{theorem}

\begin{proof}
We again use the uniform measure on $\partial B^n$. As in the previous proof, the right hand side must be the ratio of the total measure of $\partial B^n$ and the maximum density of the projection of this measure to $\mathbb R^{n-m}$. The latter density is given by (as the reader can check by elementary integration):
\[
\rho_m (x) = \frac{2\pi^{m/2}}{\Gamma(m/2)} (1-|x|^2)^{m/2-1}.
\]
\end{proof}

For $m=2$ the above theorem gives the precise estimate for a \emph{non-fractional} covering by $2$-dimensional Euclidean cylinders, extending the original proof of Moese:

\begin{corollary}
\label{corollary:2-cylinders}
If the Euclidean unit ball is covered by $2$-dimensional Euclidean cylinders $Z_1,\ldots, Z_N$ then
\[
\sum_{i=1}^N \sigma_{n-2}(Z_i) \ge \vol_{n-2}(B^{n-2}) = \frac{\pi^{n/2-1}}{\Gamma(n/2)}.
\]
\end{corollary}

A similar result for $m=1$ was proved by K\'aroly Bezdek and Alexander Litvak in~\cite[Theorem~3.1]{bezdek2009covering}.

In~\cite{kadets2004weak} Vladimir~Kadets showed the following. For any $\varepsilon>0$, there is a covering of the Hilbert space by Hilbert cylinders $Z_i$ isometric to $B_i\times H$, where $B_i$ is a $3$-dimensional ball, such that 
\[
\sum_i \sigma(Z_i) = \sum_i \vol_3 (B_i) < \varepsilon.
\]
Here $\sigma$ denotes the $3$-dimensional cross-section. In~\cite{bezdek2009tarski} K\'aroly~Bezdek seemingly thought that from Kadets' construction it follows that, for any $\varepsilon>0$, there exists sufficiently large~$n$ such that the unit ball $B^n$ can be covered by $(n-3)$-cylinders $Z_i$ with sum of their cross-section less than $\vol(B^3)$. Actually it does not follow from Kadets' construction, which was essentially based on infinite-dimensional properties of the Hilbert space. So~\cite[Problem~3.5]{bezdek2009tarski} can be restated in the following natural form:

\begin{conjecture}
If the Euclidean unit ball $B^n$ is covered by $m$-dimensional Euclidean cylinders $Z_1,\ldots, Z_N$ then
\[
\sum_{i=1}^N \sigma_{n-m}(Z_i) \ge \vol_{n-m}(B^{n-m}) = \frac{\pi^{(n-m)/2}}{\Gamma((n-m)/2+1)}.
\]
\end{conjecture} 

\subsection{Fractional covering by almost parallel planks}

Now, in addition to the above elementary considerations, we prove a result about covering of a Euclidean ball using symplectic methods. Informally, it shows that the constant $W$ from the above discussion gets closer to $2$ when the planks are ``almost parallel'':

\begin{theorem}
\label{theorem:frac-bang}
Let $\{P_i\}$ be a family of planks, covering every point of the Euclidean ball $B\subset\mathbb R^d$ at least $k$ times. Assume also that the normals to the planks $n_i$ may be oriented so that for every $i\neq j$, $n_i\cdot n_j \ge C$, where $C\ge 0$ is a constant. Then
\[
\sum_i w(P_i) \ge 2\sqrt{((k-1)C+1)k}.
\]
\end{theorem}

\begin{proof}
For every $P_i$ we consider the function $F_i$ with $dF_i = n_i$ on $P_i$ and $dF_i=0$ outside $P_i$. These functions are not smooth but the following argument remains valid after a suitable smoothening of every $F_i$. 

Then we consider the sum $F(x) = \sum_i F_i(x)$. The assumption $n_i\cdot n_j\ge C$ implies that $|dF(x)| \ge \sqrt{((k-1)C+1)k}$. Then using the displacement energy of $B\times B$ as in the proof of Theorem~\ref{theorem:almost-par}, or using Theorem~\ref{theorem:diff-on-ball}, we obtain:
\[
\max_B F(x) - \min_B F(x) \ge 2\sqrt{((k-1)C+1)k}.
\]
But the difference of every summand $F_i$ is at most $w(P_i)$, and the result follows.
\end{proof}

\bibliography{../Bib/karasev}
\bibliographystyle{abbrv}
\end{document}